\documentclass[11pt]{amsart}
\usepackage{amsmath, amssymb}
\theoremstyle{plain}
\newtheorem{thm}{Theorem}[section]
\newtheorem{theorem}[thm]{Theorem}

\newtheorem{proposition}[thm]{Proposition}
\theoremstyle{definition}
\newtheorem{remark}[thm]{Remark}

\newtheorem{definition}[thm]{Definition}
\newtheorem{claim}[thm]{Claim}

\newtheorem{question}[thm]{Question}

\numberwithin{equation}{section}

\newcommand{\sO}{{\mathcal O}}

\newcommand{\C}{{\mathbb C}}

\newcommand{\BP}{{\mathbb P}}
\newcommand{\Q}{{\mathbb Q}}
\newcommand{\R}{{\mathbb R}}

\newcommand{\Z}{{\mathbb Z}}

\title [A queston of Doctor Malte Wandel]{A question of Doctor Malte Wandel 
on automorphisms of the punctual Hilbert schemes of K3 surfaces} 
\author{Keiji Oguiso} 
 
\address{Department of Mathematics, Osaka University, Toyonaka 560-0043, Osaka, Japan and Korea Institute for Advanced Study, Hoegiro 87, Seoul, 133-722, Korea}
 \email{oguiso@math.sci.osaka-u.ac.jp}   
       
\thanks{The author is supported by JSPS Grant-in-Aid (S) No 25220701,
JSPS Grant-in-Aid (S) No 22224001, JSPS Grant-in-Aid (B) No 22340009, and by KIAS Scholar Program.}
\dedicatory{Dedicated to Professor Doctor Fabrizio Catanese on the occasion of his sixty-fifth birthday.}
\begin{document}

\maketitle

\begin{abstract} We present a sufficient condition for the punctural Hilbert scheme of length two of a K3 surface with finite automorphism group to have automorphism group of infinite order in geometric terms (Theorem \ref{geometric}). 
We then give concrete examples (Theorem \ref{main}). We also discuss about Mori dream space (MDS) structures under an extermal crepant resolution (Theorems 
\ref{main}, \ref{mds}, \ref{elementary}) from the viewpoint of automorphisms. 
These 
affirmatively answer a question of Doctor Malte Wandel.
\end{abstract}

\section{Introduction}

Thoughout this note, we work over the complex number field $\C$. We call a $\C$-valued point $P$ of a variety $V$ simply a point and denote by $P \in V$ instead of $P \in V(\C)$. Topology of a complex variety is assumed to be the Euclidean topology of the underlying analytic space. We denote the group of biregular automorphisms of $V$ by ${\rm Aut}\, (V)$ and the group of birational automorphisms of $V$ by ${\rm Bir}\, (V)$. 

Let $S$ be a smooth projective surface. We denote by $S^{[n]}$ the Hilbert scheme ${\rm Hilb}^n(S)$ of $0$-dimensional closed subschemes of length $n \ge 2$ on $S$ and by $S^{(n)}$ the Chow variety $S^n/\Sigma_n$ of $0$-dimensional cycles of length $n \ge 2$ on $S$. Here $S^n$ is the $n$-times self-product of $S$ 
and $\Sigma_n$ is the symmetric group of $n$-letters which acts on $S^n$ by $(P_i)_{i=1}^{n} \mapsto (P_{\sigma(i)})_{i=1}^{n}$. 

Let $S$ be a projective K3 surface, i.e., a smooth, simply-connected projective surface $S$ with $\sO_S(K_S) \simeq \sO_S$. Here $K_S$ is a canonical divisor of $S$. Then $S^{[n]}$ is a projective hyperk\"ahler manifold, i.e., a smooth, simply-connected projective variety with everywhere non-degenerate global regular $2$-form $\sigma_{S^{[n]}}$ such that $H^0(S^{[n]}, \Omega_{S^{[n]}}^2) = \C \sigma_{S^{[n]}}$ (\cite{Fu83}, \cite{Be84}). 

The aim of this note is to give an affirmative answer to the following question asked by Doctor Malte Wandel (Theorem
\ref{main}, See also Theorems \ref{geometric}, \ref{elementary}, \ref{mds}):

\begin{question}\label{wandel}
Is it possible that $\vert {\rm Aut}\, (S^{[2]}) \vert = \infty$ for a projective K3 surface $S$ with $\vert {\rm Aut}\, (S) \vert < \infty$?
\end{question}

Note that ${\rm Aut}\, (S)$ naturally and faithfully acts on $S^{[n]}$. Hence $\vert {\rm Aut}\, (S^{[n]}) \vert = \infty$ if $\vert {\rm Aut}\, (S) \vert = \infty$. 

Throughout this note, we denote by $\Lambda$ the even hyperbolic lattice of rank $2$ and of discriminant $17$, defined by:
$$\Lambda := (\Z^2 = \Z e_1 + \Z e_2, ((e_i,e_j)) = \left(\begin{array}{rr}
2 & 5\\
5 & 4\\
\end{array} \right))\,\, , \,\, |{\rm det}\, \left(\begin{array}{rr}
2 & 5\\
5 & 4\\
\end{array} \right))| = 17 .$$
Here, by a lattice $L = (L, (*,**))$, we mean a pair of a free $\Z$-module $L$ of finite rank $r$ and a non-degenerate $\Z$-valued symmetric bilinear form $(*,**)$ on $L$. $L$ is called hyperbolic if $(*,**) \otimes \R$ is 
of singnature $(1, r-1)$, and even if $(x, x) \in 2\Z$ for all $x \in L$. Lattice isomorphisms are defined by an obvious manner. The N\'eron-Severi group ${\rm NS}\, (S)$ of a projective K3 surface $S$ is an even hyperbolic lattice of rank $\rho(S)$, the Picard number of $S$, with respect to the intersection paring $(*, **)$ (See eg. \cite[Chap.VIII]{BHPV}). 

Our main result is the following:
 
\begin{theorem}\label{main}
Let $S$ be a necessarily projective K3 surface such that ${\rm NS}\, (S) \simeq \Lambda$. Then:

(1) $\vert {\rm Aut}\, (S^{[2]}) \vert = \infty$ and $\vert {\rm Aut}\, (S) \vert < \infty$. 

(2) The Hilbert-Chow morphism 
$\mu : S^{[2]} \to S^{(2)}$
is an extremal crepant resolution such that the source $S^{[2]}$ is not a Mori dream space, but the target $S^{(2)}$ is a Mori dream space.
\end{theorem}

See also Theorem \ref{geometric} in Section 2 for a geometric criterion related to Theorem \ref{main} (1). The essential point of the proof of Theorem \ref{main} (1) is as follows. $S$ in Theorem \ref{main} has at least two different embeddings into $\BP^3$. The images are smooth quratic surfaces with no line. From this, we see that $S^{[2]}$ has at least two different biregular involutions, called Beauville's involutions (\cite[Sect. 6]{Be83}), corresponding to these two embeddings. See Sections 2 and 3 for details. 

The notion of Mori dream space (MDS) was introduced by \cite[Def. 11]{HK00}. See also Definition \ref{md} in Section 5. Theorem \ref{main} (2) was kindly suggested by Professor Shinnosuke Okawa. See also Theorem \ref{mds} and Theorem \ref{elementary} in Sections 5 and 4 for slightly more general results related to Theorem \ref{main} (2). 

\begin{remark}\label{moduli}
K3 surfaces $S$ in Theorem (\ref{main}) form a dense subset of the $18$-dimensional family of $\Lambda$-polarized K3 surfaces (\cite{Do96}, See also \cite[Cor. 2.9]{Mo84} for the existence). 
\end{remark}

\begin{remark}\label{picardnumber}

(1) If $S$ is a K3 surface with $\rho(S) = 1$, then both ${\rm Aut}\, (S^{[n]})$ and ${\rm Bir}\, (S^{[n]})$ are finite groups 
(See eg. \cite[Cor. 5.2]{Og12}). By the definition, $\rho(S) = 2$ for K3 surfaces $S$ in Theorem \ref{main}.  
Thus, $\rho(S) =2$ is the {\it smallest} Picard number for Question (\ref{wandel}) to be affirmative.

(2) Theorem \ref{geometric} in Section 2 and our proof of Theorem \ref{main} (1) in Section 3 suggest that there will be many other ways to construct examples similar to the ones in Theorem \ref{main}. 
Compare with \cite[Sect. 3]{MN02}. 
\end{remark}

We close Introduction by recalling the following question of Professor Alessandra Sarti:

\begin{question}\label{sarti}
Is ${\rm Aut}\, (X^{[n]}) = {\rm Aut}\, (X)$
for any projective K3 surface $X$ and any $n \ge 3$?
\end{question}

\begin{remark}\label{yoshioka}
If Question \ref{sarti} would be affirmative for our $S$ in Theorem \ref{main}, then ${\rm Aut}\, (S^{[n]}) = \infty$ only when $n = 2$. A  work of Markman and Yoshioka (\cite{MY14}) might be of 
some help to check this. Professor Brendon Hessett also kindly informed me that he found a counterexample for Question \ref{sarti} in $n=3$. 
\end{remark}

{\bf Ackowledgement.} I would like to express my best thanks to Doctor Malte Wandel for his question and to Professor Shinnosuke Okawa for discussions at the 
conference, 
"Arithmetic and Algebraic Geometry" held at University of Tokyo, January 2014, from which this note was grown up. I would like to express my sincere thanks to Professors Alice Garbagnati, Fabrizio Catanese, H\'el\`ene Esnault, Brenton Hassett, Yujiro Kawamata, Thomas Peternell and Alessandra Sarti 
for valuable suggestions, discussions and encouragement. I would like to express my thanks to Professors Toshiyuki Katsura, Iku Nakamura and Tomohide Terasoma for invitation to the above mentioned conference. It is my great honnor to dedicate this note to Professor Doctor Fabrizio Catanese on the occasion of his sixty-fifth birthday, from whom I have learned much about mathematics since 1993. 

\section{A geometric criterion}

In this section, we shall prove the following theorem. Our proof of Theorem \ref{main} (1) will be reduced to this theorem in Section 3:

\begin{theorem}\label{geometric}
Let $S$ be a projective K3 surface of Picard number $\rho(S) = 2$ such that:

(1) $S$ has either a smooth elliptic curve or a smooth rational curve; 

(2) $S$ has very ample divisors $H_1$, $H_2$ such that 
$$(H_1,H_1) = (H_2,H_2) = 4\,\, , \,\, H_1 \not= H_2$$
 in ${\rm NS}\, (S)$; and 

(3) $S$ has no smooth rational curve $C_k$ such that $(C_k,H_k) = 1$ for each $k =1$, $2$.

Then $\vert {\rm Aut}\, (S) \vert < \infty$ and $\vert {\rm Aut}\, (S^{[2]}) \vert = \infty$.
\end{theorem}

\begin{proof} By (1), the nef cone $\overline{{\rm Amp}}\, (S) \subset {\rm NS}\, (S) \otimes \R$ has at least one boundary ray defined over $\Q$. Since $\rho(S) = 2$, this implies $\vert {\rm Aut}\, (S) \vert < \infty$ (see eg. \cite[Prop. 2.4]{Og12}). 

By (2), we have 
two embeddings
$$\Phi_k := \Phi_{\vert H_k \vert} : S \rightarrow \BP^3\,\, 
$$
associated to the complete linear systems $|H_k|$ ($k = 1$, $2$).
Let $S_k := \Phi_k(S)$. Then $S_k \subset \BP^3$ is a smooth quartic surface by (2) and $S_k$ contains no line by (3). 

Let $P, Q \in S$ be general points of $S$ and $k \in \{1, 2\}$. Then, the projective line $\Phi_k(P)\Phi_k(Q)$ in $\BP^3$ meets $S_k$ in four points, namely, $\Phi_k(P)$, $\Phi_k(Q)$, $\Phi_k(R)$, $\Phi_k(S)$. The correspondence $\{P, Q\} \mapsto \{R, S\}$ then defines a birational automorphism of $S^{[2]}$ of order $2$: 
$$\iota_k : S^{[2]} \cdots \to S^{[2]}\,\, .$$
$\iota_k$ is called the {\it Beauville involution} with respect to $H_k$ (\cite[Prop. 11]{Be83}). As $S_k \subset \BP^3$ contains no projective line, $\iota_k$ is biregular, i.e., $\iota_k \in {\rm Aut}\, (S^{[2]})$ (\cite[ibid.]{Be83}). 

Recall that ${\rm NS}\, (S^{[2]})$ is a hyperbolic lattice of rank $3 = \rho(S) +1$ with respect to the Beauville-Bogomolov form (\cite[Prop. 6]{Be84}). More strongly, we have the following natural identification as lattices (\cite[ibid.]{Be84}):  
$${\rm NS}\, (S^{[2]}) = {\rm NS}\, (S) \oplus \Z e\,\, , \,\, (e,e) = -2\,\, 
.$$ 
Here $e = E/2$ for the exceptional divisor $E$ of the Hilbert-Chow morphism   
$\nu : S^{[2]} \rightarrow S^{(2)}$ (\cite[ibid.]{Be84}). 
Note that $H_1$ and $H_2$ are linearly independent in ${\rm NS}\, (S)$. Hence we have
$${\rm NS}\, (S^{[2]}) \otimes \Q = 
\Q \langle H_1 -e, H_2 -e, e \rangle\,\, .$$

Consider the action of the Beauville involution $i_k$ on 
${\rm NS}\, (S^{[2]})$:
$$\iota_k^* : {\rm NS}\, (S^{[2]}) \rightarrow {\rm NS}\, (S^{[2]})\,\, .$$ 
This is the {\it anti-involution} with respect to the invariant vector 
$H_k -e$:
$$\iota_k^{*} (x) = -x + (x, H_k -e)(H_k -e)\,\, .$$
Here we note that $(H_k -e, H_k-e) = 2$. 
This formula is due to Debarre (\cite[Thm. 4.1]{De84}, see also \cite[Prop. 4.1]{OGr05}). 
Then, with respect to the $\Q$-basis 
$\langle H_1 -e, H_2 -e, e \rangle$ of ${\rm NS}\, (S^{[2]}) \otimes \Q$, we have the following matrix representations of $\iota_k^*$:
$$\iota_1^* = \left(\begin{array}{rrr}
1 & m-2 & 2\\
0 & -1 & 0\\
0 & 0 & -1\\
\end{array} \right)\,\, , \,\, \iota_2^* = \left(\begin{array}{rrr}
-1 & 0 & 0\\
m-2 & 1 & 2\\
0 & 0 & -1\\
\end{array} \right)\,\, .$$
Here we put $m := (H_1,H_2)$. As $m$ is a positive integer and 
$$m^2 = (H_1,H_2)^2 > (H_1,H_1)(H_2,H_2) = 4^2$$
by the Hodge index theorem, it follows that $m \ge 5$. 
Observe then that 
$$(\iota_2\iota_1)^* = \iota_1^*\iota_2^* = \left(\begin{array}{rrr}
(m-2)^2-1 & m-2 & 2m-6\\
-(m-2) & -1 & -2\\
0 & 0 & 1\\
\end{array} \right)\,\, .$$
Let $F(t)$ be the characteristic polynomial of $(\iota_2\iota_1)^*$. 
Then
$$F(t) = f(t)(t-1)\,\, ,$$
where 
$$f(t) = t^2 - ((m-2)^2 -2)t + 1\,\, .$$ 
By $m \ge 5$, we have
$$(m-2)^2 - 2 \ge (5-2)^2 -2 \ge 7\,\, ,\,\, 
D := ((m-2)^2 -2)^2 - 4 > 0\,\, .$$ 
Hence the quadratic polynomial $f(t)$ has two distinct real roots. We denote 
them by $\alpha < \beta$. As
$$2\beta > \alpha + \beta = (m-2)^2 -2 \ge 7\,\, ,$$
it follows that $\beta > 1$. Let $n$ be an integer. Then $\beta^n$ is an eigenvalue of $((\iota_2\iota_1)^*)^n$. As $\beta > 1$, it follows that $n=0$ if $\beta^n = 1$. 
Thus $(\iota_2\iota_1)^*$ is of infinite order. Hence $\iota_2\iota_1 \in {\rm Aut}\, (S^{[2]})$ 
is of infinite order as well. This completes the proof of Theorem \ref{geometric}. 
\end{proof}

\section{Proof of Theorem (\ref{main})(1)}

In this section, we shall prove Theorem (\ref{main}) (1).  
We proceed by dividing into several steps. Set
$${\rm NS}\, (S) = \Z \langle L, H \rangle\,\, ,\,\, (L,L) = 2\,\, ,\,\, (H,H) = 4\,\, ,\,\, (L,H) = 5\,\, .$$

\begin{claim}\label{elliptic}
There is no $C \in {\rm NS}\, (S) \setminus \{0\}$ such that $(C,C) = 0$. In particular, $S$ has no curve $C$ with arithmetic genus $p_a(C) = 1$. 
\end{claim}

\begin{proof}
Write $C = xL + yH$ in ${\rm NS}\, (S)$. Here $x, y \in \Z$. Then 
$$(C,C) = 2x^2 + 10xy + 4y^2\,\, .$$
So, if $(C,C) = 0$, then $x = (-5 \pm \sqrt{17})y/2$. As $x, y \in \Z$ and $(-5 \pm \sqrt{17})/2 \not\in \Q$, it follows that $y = 0$ and therefore $x=0$. Hence $C = 0$ in ${\rm NS}\, (S)$. 
\end{proof}

\begin{claim}\label{norational}
There is no $C \in {\rm NS}\, (S)$ such that $(L,C) = 0$ and $(C,C) = -2$. 
In particular, there is no smooth rational curve $C \subset S$ such that $(L, C) = 0$. 
\end{claim}

\begin{proof}
Write $C = xL + yH$ in ${\rm NS}\, (S)$. Here $x, y \in \Z$. Then 
$$(L,C) = 2x + 5y\,\, .$$
If $(L,C) = 0$, then there is $z \in \Z$ such that $x = 5z$ and $y = -2z$. Then
$$(C,C) = (5L-2H, 5L -2H)z^2 = -34z^2 \not= -2\,\, .$$ 
This implies the first assertion of Claim \ref{norational}. For the last assertion, we note that if $C$ is a smooth rational curve on $S$, then $(C,C) = -2$ by the adjunction formula. 
\end{proof}

Set $P(S) := \{x \in {\rm NS}\, (S) \otimes \R \, 
\vert\, (x,x) > 0\}$ 
and denote by $P^{+}(S)$ the $\Q$-rational hull of $P(S)$ in ${\rm NS}\, (S) \otimes \R$. Let ${\rm Amp}\,(S) \subset {\rm NS}\, (S) \otimes \R$ be the ample cone of $S$ and ${\rm Amp}^{+}\, (S)$ be the $\Q$-rational hull of ${\rm Amp}\,(S)$ in ${\rm NS}\, (S) \otimes \R$. Note that ${\rm Amp}^{+}\, (S) \subset P^{+}(S)$. Let $W({\rm NS}\, (S))$ be the group generated 
by the reflections on ${\rm NS}\, (S)$: 
$$r_C : x \mapsto x + (x,C)C\,\, .$$
Here $C$ runs through the all elements of ${\rm NS}\, (S)$ such that $(C,C) = -2$.
The product group $W({\rm NS}\, (S)) \times \{\pm id_{{\rm NS}\, (S)}\}$ acts on 
both $P^{+}(S)$ and $P^{+}(S) \cap {\rm NS}\, (S)$. 

\begin{claim}\label{fundamentaldomain}
${\rm Amp}^{+}\, (S)$ is the fundamental domain of the action of $W({\rm NS}\, (S)) \times \{\pm id_{{\rm NS}\, (S)}\}$ on $P^{+}(S)$. 
\end{claim}

\begin{proof}
This is a version of the Nakai-Moishezon criterion for the ampleness of line bundles on a projective K3 surface (See eg., \cite[Chap. VIII]{BHPV}). 
\end{proof}

\begin{claim}\label{ample}
We may (and will) assume that $L$ is ample.
\end{claim}

\begin{proof}
By Claim \ref{fundamentaldomain}, we may and will assume that $L \in {\rm Amp}^{+}\, (S) \cap {\rm NS}\, (S)$. Then $L$ is nef and big. Thus $L = K_S + L$ is semi-ample by the base point free theorem (\cite{Ka84}). So, if $L$ would not be ample, then there would be an irreducible curve $C \subset S$ such that $(L, C) = 0$. By the Hodge index theorem, $(C, C) < 0$. However, then $p_a(C) =0$ by the adjunction formula, and $C$ would be a smooth rational curve on $S$ with $(L,C) = 0$, a contradiction to Claim \ref{norational}. 
\end{proof}

{\it From now on, we assume that $L \in {\rm NS}\, (S) \simeq {\rm Pic}\, (S)$ is an ample class.}

\begin{claim}\label{free}
The complete linear system $\vert L \vert$ is free and the associated morphism 
$$\Phi := \Phi_{\vert L \vert} : S \to \BP^2$$
is a finite double cover. 
\end{claim}

\begin{proof}
We prove first that the complete linear system $\vert L \vert$ is free. Consider the decomposition 
$$\vert L \vert = \vert M \vert + F\,\, .$$ 
Here $\vert M \vert$ is the movable part and $F$ is the fixed component of $|L|$. If $|L|$ would not be free, then $F$ would be effective and {\it non-zero} by \cite[Cor. 3.2]{SD74}. Then 
$$2 = (L,L) > (L,M) \ge (M,M) \ge 0\,\, .$$ 
As $(M,M) \in 2 \Z$, it would follow that $(M,M) = 0$, a contradiction to Claim \ref{elliptic}.  Hence $|L|$ is free. 
As $(L,L) =2$ and $L$ is ample, we have $h^0(S, L) = 3$ by the Riemann-Roch formula. Hence $\Phi_{\vert L \vert}$ is a finite double covering of $\BP^2$ 
as claimed.  
\end{proof}
Let $\tau \in {\rm Bir}\, (S)$ be the covering involution of $\Phi : S \rightarrow \BP^2$. 
\begin{claim}\label{involution}
$\tau \in {\rm Aut}\, (S)$ and 
$$\tau^* = \left(\begin{array}{rr}
1 & 5\\
0 & -1\\
\end{array} \right)\,\,$$
on ${\rm NS}\,(S)$ with respect to the $\Z$-basis $\langle L, H \rangle$ of ${\rm NS}\, (S)$. 
\end{claim}

\begin{proof} 
Note that ${\rm Bir}\, (S) = {\rm Aut}\, (S)$ as $S$ is a smooth projective surface with nef $K_S$. Hence $\tau \in {\rm Aut}\, (S)$. By the definition, $L = \Phi^*{\mathcal O}_{{\mathbf P}^2}(1)$. 
Thus $\tau^*L = L$. 

We have
$$\Phi^*\Phi_* H = H + \tau^*H\,\, .$$
Here $\Phi_* H$ is the pushforward as $1$-cycles. Then $\Phi_* H \in \vert {\mathcal O}_{\BP^2}(m) \vert$ for some $m \in \Z_{>0}$. Then, by the projection formula, we have
$$(\Phi^*\Phi_* H,L) = (\Phi^*\Phi_* H, \Phi^*{\mathcal O}_{\BP^2}(1)) 
= 2(\Phi_* H,{\mathcal O}_{\BP^2}(1)) = 2m\,\, .$$
On the other hand, as $\Phi^*\Phi_* H = H + \tau^*H$ and $\tau^*L = L$, we also have
$$(\Phi^*\Phi_* H, L) = (H+ \tau^*H, L) = (H,L) + (\tau^*H,L) = (H,L) + 
(\tau^*H , \tau^*L) = 2(H,L) = 10\,\, .$$
Hence $m = 5$ and $H + \tau^*H = \Phi^{*}({\mathcal O}_{\BP^2}(5))$. As $L = \Phi^*{\mathcal O}_{{\mathbf P}^2}(1)$, it follows that $\tau^*H = 5L -H$. Combining this with 
$\tau^*L = L$, we obtain the last assertion.
\end{proof}

\begin{claim}\label{trivial}
If $C$ is a smooth rational curve on $S$, then $\vert C \vert = \{C\}$.
\end{claim}

\begin{proof} 
This is well-known and easily follows from the irreducibility of $S$ and 
the fact that $(C,C) = -2 <0$ (See eg. \cite[Chap. VIII]{BHPV}).
\end{proof}

\begin{claim}\label{firstcondition}

(1) $S$ has exactly two smooth rational curves and their classes in ${\rm NS}\, (S)$ are
$$-L + 2H\,\, ,\,\, \tau^*(-L+2H) = 9L - 2H\,\, .$$
(2) The cone of effective curves on $S$ is
$$\overline{{\rm NE}}(S) = \R_{\ge 0}(-L + 2H) + \R_{\ge 0}(9L - 2H)\,\, .$$
(3) $S$ satisfies the condition (1) in Theorem \ref{geometric}. 
\end{claim}

\begin{proof} 
The assertion (3) follows from (2). Recall that $\rho(S) =2$. 
Then (2) follows from (1) and Claim \ref{trivial}. The fact that $\tau^*(-L +2H) = 9L - H$ follows from Claim \ref{involution}. By Claim \ref{trivial}, it now suffices to show that $\vert -L + 2H \vert$ contains the class of a smooth rational curve. Set 
$$C := -L + 2H$$
in ${\rm NS}\, (S)$. Then 
$$(C,C) = -2\,\, ,\,\, (C,L) = 8 > 0\,\, .$$
Hence, by the Riemann-Roch formula and the Serre duality, it follows that $\vert C \vert \not= \emptyset$. Note here that $L$ is ample. As $(C,C) = -2$, we can then choose a smooth rational curve $C_1$ and an effective curve $C_2$, possibly $0$, such that 
$$C_1 + C_2 \in \vert C \vert\,\, .$$
As $(L,C) = (L,-L+2H) = 8$ and $L$ is ample, it follows that 
$$1 \le k \le 8\,\, .$$
Here we put $k := (L,C_1)$. 
Consider the sublattice 
$$M := \Z \langle L, C_1 \rangle \subset {\rm NS}\, (S)\,\, .$$
The intersection matrix of $M$ with respect to the $\Z$-basis $\langle L, C_1 \rangle$ of $M$ 
is
$$\left(\begin{array}{rr}
2 & k\\
k & -2\\
\end{array} \right)\,\, .$$
The discriminant of $M$ is then $k^2 +4$. On the other hand, the discriminant of ${\rm NS}\, (S)$ is $17$. Thus, by the elementary divisor theorem, we have 
$$k^2 + 4 = 17 l^2$$ 
for some $l \in \Z$. As $l$ and $k$ are integers such that $1 \le k \le 8$, it follows that $k = 8$. Hence $(L,C_2) = 0$. Thus $C_2 = 0$, as $L$ is ample and $C_2$ is effective. Therefore $C = -L +2H$ is the class of a smooth rational curve 
$C_1$.
\end{proof}

From now, we shall show that $H$ and $\tau^*H$ satisfy the conditions (2) and (3) in Theorem \ref{geometric}. We proceed by dividing into several steps.

\begin{claim}\label{ampleh}
$H$ is ample and $h^0(S, H) = 4$.
\end{claim}

\begin{proof} 
Observe that
$$(H,-L+2H) = 3 > 0\,\, ,\,\, (H,9L-2H) = 37 > 0\,\, .$$
Thus $H$ is ample by the Kleiman's criterion and 
Claim \ref{firstcondition} (2). 
As $(H,H) = 4$, the second assertion now follows from the Riemann-Roch 
formula.
\end{proof}

\begin{claim}\label{movableh}
$\vert H \vert$ has no fixed component.
\end{claim}

\begin{proof} 
Assume to the contrary that $\vert H \vert$ has a fixed component. Then $\vert H \vert = \vert M \vert + F$, where $\vert M \vert$ is the movable part and $F$ is a non-zero effective divisor. Using Claim \ref{elliptic}, we have
$$4 = (H,H) > (H,M) \ge (M,M) > 0\,\, ,\,\, 5 = (H,L) > (M,L) > 0\,\, .$$
Therefore $(M,M) = 2$ by Claim \ref{elliptic} and by the fact that $(M, M) \in 2 \Z$. Put $k := (M,L)$. Then $1 \le k  \le 4$ 
by the second inequality above, and the discriminant of $\Z \langle L, M \rangle$ is $\vert k^2 - 4 \vert$, possibly $0$. For the same reason as in Claim \ref{firstcondition}, we have 
$$\vert k^2 - 4 \vert = 17 l^2$$ for some $l \in \Z$. As $l$ and $k$ are integers such that $1 \le k \le 4$, it follows that $k =2$ and therefore $\vert k^2 - 4 \vert = 0$. Thus $M = mL$ for some $m \in \Q$ by the Hodge index theorem. Then, by $(M,M) = 2 = (L,L)$, it follows that $M = L$ in ${\rm Pic}\, (S) \simeq {\rm NS}\, (S)$, a contradiction to the fact that $h^0(S, M) = h^0(S, H) = 4$ and $h^0(S, L) = 3$. 
\end{proof}

\begin{claim}\label{veryampleh}
$H$ is very ample.
\end{claim}

\begin{proof}
By \cite[Theorem 5.2]{SD74} and by Claim \ref{movableh}, it suffices to check the following two facts:

(i) there is no irreducible curve $E \subset S$ such that $p_a(C) = 1$ and $(E, H) =2$;

(ii) there is no irreducible curve $B \subset S$ such that $p_a(B) = 2$ and $H = 2B$ in ${\rm NS}\, (S)$.

The assertion (i) follows from Claim \ref{elliptic}. By the adjunction formula, $(B, B) =2$ if $p_a(B) = 2$. The assertion (ii) follows 
from this equality and $(H,H) =4$. 
\end{proof}

\begin{claim}\label{line}
There is no $C \simeq \BP^1$ on $S$ such that $(H,C) = 1$.
\end{claim}

\begin{proof}
Recall that the class of $C$ is either $-L +2H$ or $\tau^*(-L + 2H) = 9L -2H$ 
by Claim \ref{firstcondition}.  The result now follows from
$$(H, -L+2H) = 3 \not= 1\,\, ,\,\, (H, 9L-2H) = 37 \not= 1\,\, .$$
\end{proof}
Put $H_1 := H$ and $H_2 := \tau^* H$. Then $H_2 = 5L -H$ by Claim \ref{involution}.
\begin{claim}\label{secondcondition}
$H_1$ and $H_2$ satisfy the conditions (2) and (3) in Theorem \ref{geometric}. 
\end{claim}

\begin{proof}
As $H_1 \not= H_2$ in ${\rm NS}\, (S)$, $(H,H) = 4$ and $\tau \in {\rm Aut}\, (S)$, the result follows from Claims \ref{veryampleh} and \ref{line}. 
\end{proof}

Theorem \ref{main} (1) now follows from Theorem \ref{geometric} and Claims \ref{firstcondition} and \ref{secondcondition}. 

\section{Some properties of the product of K3 surfaces}

Let $V$ be a normal $\Q$-factorial projective variety. We denote by $\omega_V$ 
the dualizing sheaf of $V$. A divisor $D$ is called movable if there is a positive integer $k$ such that the complete linear system $|kD|$ has no fixed component. A divisor class $d \in {\rm NS}\, (V)$ is called movable if $d$ is represented by a movable divisor $D$.  
By ${\rm Mov}\, (V) \subset {\rm NS}\, (V) \otimes \R$, we denote the movable cone of $M$, i.e., the convex cone generated by the movable divisor classes. We denote by $\overline{\rm Mov}\, (V)$ the closure of the movable cone in the linear space ${\rm NS}\, (V) \otimes \R$. 

By ${\rm Amp}\, (V) \subset {\rm NS}\, (V) \otimes \R$, we denote the ample cone of $V$, i.e., the open convex cone generated by the ample divisor classes. We denote by $\overline{\rm Amp}\, (V)$ the closure of the ample cone in ${\rm NS}\, (V) \otimes \R$. The cone $\overline{\rm Amp}\, (V)$ is called the nef cone. 

Let $n$ be a positive integer, possibly $n=1$, and $S_i$ ($1 \le i \le n$) be mutually non-isomorphic projective K3 surfaces. Let $m_i \ge 1$ ($1 \le i \le n$) be positive integers (not necessarily distinct). We consider the following product: 
$$M := S_1^{m_i} \times S_2^{m_2} \times \cdots \times S_n^{m_n}\,\, .$$
$M$ is a smooth projective variety. We denote the tangent bundle of $M$ by 
$T_M$. 
We denote the $(i,j)$-th factor of $M$ by $S_{ij}$ and by
$$p_{ij} : M \to S_{ij}$$
the projection to the $(i,j)$-th factor. Here 
$$1 \le i \le n\,\, ,\,\, 1 \le j \le m_i\,\, .$$
Note that $S_{ij} = S_i$ as abstract varieties. We identify
$${\rm NS}\, (S_{ij}) = p_{ij}^*{\rm NS}\, (S_{ij}) \subset {\rm NS}\, (M)\,\, .$$
We denote by $\Sigma_k$ the symmetric group of $k$-letters. Then the group 
${\rm Aut}\, (S_i)^{m_i} \rtimes \Sigma_{m_i}$ naturally acts on $S_i^{m_i}$ 
and hence on $M$. More precisely, by the identification
$${\rm Aut}\, (S_i)^{m_i} \rtimes \Sigma_{m_i} = 
\prod_{k= 1}^{i-1} \{ id_{S_k^{m_k}} \} \times ({\rm Aut}\, (S_i)^{m_i} \rtimes \Sigma_{m_i}) \times \prod_{k= i+1}^{n} \{ id_{S_k^{m_k}} \}\,\, ,$$
we regard the group ${\rm Aut}\, (S_i)^{m_i} \rtimes \Sigma_{m_i}$ as a subgroup of ${\rm Aut}\, (M)$. 

In this section, we prove the following:

\begin{theorem}\label{elementary}

(1) $\omega_M \simeq \sO_M$, $H^0(M, T_M) = 0$, $H^1(M, \sO_M) = 0$ and ${\rm Pic}\, (M) = {\rm NS}\, (M)$ under the first Chern class map.

(2) Under the identification 
${\rm NS}\, (S_{ij}) = p_{ij}^*{\rm NS}\, (S_{ij})$ made above, 
$${\rm Pic}\, (M) = {\rm NS}\, (M) = \prod_{i=1}^{n} \prod_{j=1}^{m_i} {\rm NS}\, (S_{ij})\,\, .$$

(3) Under the identification (2),
$$\overline{\rm Mov}\, (M) = \overline{\rm Amp}\, (M) = \prod_{i=1}^{n} \prod_{j=1}^{m_i} \overline{{\rm Amp}}\, (S_{ij})\,\, .$$

(4) Under the inclusion ${\rm Aut}\, (S_i)^{m_i} \rtimes \Sigma_{m_i} \subset {\rm Aut}\, (M)$ explained above, 
$${\rm Bir}\, (M) = {\rm Aut}\, (M) = \prod_{i=1}^{n} ({\rm Aut}\, (S_i)^{m_i} \rtimes \Sigma_{m_i})\,\, .$$
\end{theorem}

\begin{proof}
The first assertion of (1) follows from the fact that $\omega_{S_i} \simeq \sO_{S_i}$. The second (resp. the third) assertion of (1) follows from the fact that $h^0(T_{S_i}) = 0$ (resp. $h^1(\sO_{S_i}) = 0$) and the K\"unneth formula. The last assertion then follows from the exponential sequence. 

Let us prove (2). The first equality of (2) follows from (1). As $H^1(S_{ij}, \Z) = 0$, we have 
$$H^2(M, \Z) = \prod_{i=1}^{n} \prod_{j=1}^{m_i} p_{ij}^{*} H^2(S_{ij}, \Z)\,\, ,$$
$$H^{1,1}(M, \C) = \prod_{i=1}^{n} \prod_{j=1}^{m_i} p_{ij}^{*} H^{1,1} (S_{ij}, \C)$$
by the K\"unneth formula. Thus the second equality follows from the Lefschetz $(1,1)$-theorem. 

Let us prove (3). Choose points $Q_{ij} \in S_{ij}$ ($1 \le i \le n$, $1 \le j \le m_i$). Let 
$$\iota_{ij} : S_{ij} \to M$$ 
be the inclusion defined by
$$x \mapsto (Q_{11}, Q_{12}, \cdots , Q_{i(j-1)}, x, Q_{i(j+1)}, 
\cdots , Q_{nm_n})\,\, .$$ Note that $p_{ij} \circ \iota_{ij} = id_{S_{ij}}$. 
Let $h \in {\rm NS}\, (M)$. Under the equality of the N\'eron-Severi groups in (2), we write 
$$h = (h_{11}, h_{12}, \cdots , h_{ij}, \cdots , h_{nm_n})\,\, .$$
Here $h_{ij} \in {\rm NS}\, (S_{ij})$. Then 
$h = \sum_{i, j}p_{ij}^*h_{ij}$ and $h_{ij} = \iota_{ij}^{*}h$. As the pullback of the nef class is nef, the last equality in (3) follows from these two equalities. 

By the last equality in (3), every projective birational contraction of $M$ is of the form 
$$\mu_{11} \times \mu_{12} \times \cdots \mu_{ij} \cdots \times \mu_{nm_n} : M \to V_{11} \times V_{12} \times \cdots \times V_{ij} \times \cdots \times V_{nm_n}\,\, ,$$
where $\mu_{ij} : S_{ij} \to V_{ij}$ ($1 \le i \le n$, $1 \le j \le m_i$) are birational projective contractions, in which some of $\mu_{ij}$ are possibly isomorphisms. Thus, $M$ has no small projective contraction. As $M$ is a smooth projective variety with $\omega_M \simeq \sO_M$, the first equality of (3) then follows from  \cite[Thm. 5.7, Rmk. 5.9]{Ka88}.

Let us prove (4).  As $M$ has no small projective contraction, $M$ has no flop, either. 
Then, by the fundamental result due to Kawamata (\cite{Ka08}), it follows that
${\rm Bir}\, (M) = {\rm Aut}\, (M)$. Here we also used the fact that $M$ is a smooth projective variety with $\omega_M \simeq \sO_M$ in order to apply \cite{Ka08}. 

It remains to prove the last equality in (4). It is clear that the group in the right hand side is a subgroup of ${\rm Aut}\, (M)$. We shall prove the opposite inclusion by considering special fibrations on $M$. We proceed by the induction on $k = \sum m_{ij}$. The result is clear if $k=1$. 

We call a surjective morphism $\varphi : M \to B$ a fibration if $B$ is a normal projective variety and the fibers of $f$ are connected. Note that a fibration is given by the morphism $\Phi_{|L|}$ associated to some complete, free linear system $|L|$. Indeed, $\varphi = \Phi_{|\varphi^* H|}$ for any very ample divisor $H$ on $B$, as $H^0(B, H) \simeq H^0(M, \varphi^*H)$ under $\varphi^*$, by the projection formula and the connectedness of fibers. Two fibrations $\varphi : M \to B$ and $(\varphi)' : M \to B'$ are called isomorphic if there is an isomorphism $\psi : B \to B'$ such that $\psi \circ \varphi = (\varphi)'$. 

Let $S$ be a projective K3 surface. Then a fibration on $S$ is either an isomorphism, a birational contraction onto a K3 surface with Du Val singular points, a fibration $S \to \BP^1$ whose general fiber is a smooth curve of genus one, or a surjective morphism to a point, ${\rm Spec}\, \C$. These fibrations are given respectively by sufficiently positive multiples of an ample divisor, a non-ample nef, big dvisor, a non-zero nef divisor with self-intersection number $0$ and zero divisor. Therefore, by (3), a fibration $\varphi : M \to B$ is isomorphic to one of $p_{ij} : M \to S_{ij}$ if $B$ is a smooth K3 surface.  

Let $f \in {\rm Aut}\, (M)$. Then $p_{11} \circ f : M \to S_{11}$ is a fibration. As $S_{11}$ is a smooth projective K3 surface, this fibration is isomorphic to one of $p_{ij} : M \to S_{ij}$. As $S_i \not\simeq S_1$ for $i \ge 2$ by our assumption, it follows that $i =1$. That is, there is $j$ ($1 \le j \le m_1$) such that $p_{11} \circ f : M \to S_{11}$ is isomorphic to $p_{1j} : M \to S_{1j}$. 
Choose an isomorphism 
$$g : S_{1j} \to S_{11}$$ 
such that 
$$g \circ p_{1j} = p_{11} \circ f\,\, .$$
Then $g \in {\rm Aut}\, (S_1)$, as $S_1 = S_{1j}$. 
Let $\sigma \in {\rm Aut}\, (S_1^{m_1}) \subset {\rm Aut}\, (M)$ be the automorphism defined by the permutation $(1, j) \in \Sigma_{m_1}$. By the definition, $p_{1j} \circ \sigma = p_{11}$. Then 
$$p_{11} \circ (f \circ \sigma) = (p_{11} \circ f) \circ \sigma = (g \circ p_{1j}) \circ \sigma = g \circ (p_{1j} \circ \sigma) = g \circ p_{11}\,\, .$$
Hence $f \circ \sigma \in {\rm Aut}\, (M)$ and $g \in {\rm Aut}\, (S_1)$ 
commute with the projection $p_{11}$. This means that $f \circ \sigma$ is of the form:$$M  = S_{11} \times N \ni (x, y) \mapsto (g(x), h_x(y)) \in S_{11} \times N = M\,\, .$$
Here 
$$N = \prod_{j=2}^{m_1}S_{1j} \times \prod_{i=2}^{n} \prod_{j=1}^{m_j} S_{ij}\,\, ,$$ 
$x \in S_{11}$ and $y \in N$. 
Then the correspondence $x \mapsto h_x(*)$ defines an analytic morphism 
$$\tau : S_{11} \to {\rm Aut}\, (N)\,\, ;\,\, x \mapsto h_x(*)\,\, .$$
As $H^0(N, T_N) = 0$ by (1), ${\rm Aut}\, (N)$ is a descrete group. 
Hence $\tau$ 
is a constant map. This means that $f \circ \sigma$ is of the form:
$$M  = S_{11} \times N \ni (x, y) \mapsto (g(x), h(y)) 
\in S_{11} \times N = M\,\, ,$$
for some automorphism $h \in {\rm Aut}\, (N)$, being independent of $x \in S_{11}$. 
Now, we can apply the assumption of the induction for $N$ and $h \in {\rm Aut}\, (N)$ and obtain the desired inclusion. 
\end{proof}

\section{An application for MDS and proof of Theorem \ref{main} (2)}

In this section we shall prove Theorem \ref{main} (2). We prove it  in a slightly generalized form Theorem \ref{mds}.

We first recall the notion of Mori dream space (MDS) from \cite{HK00}. 

Let $V$ be a normal $\Q$-factorial projective variety. A divisor $D$ on $V$ 
is called semi-ample if $|mD|$ is free for some positive integer $m$. We call a birational map $f : V \cdots\to V'$
a small $\Q$-factorial modification if $V'$ is a normal $\Q$-factorial projective variety and $f$ is isomorphic in codimension one, i.e., there are Zariski closed subsets $B \subset V$ and $B' \subset V'$ of codimension $\ge 2$, possibly empty, such that 
$$f |V \setminus B : V \setminus B \to V' \setminus B'$$
is an isomorphism.  
\begin{definition}\label{md}
Let $V$ be a projective variety. $V$ is called a MDS if the following three conditions are satisfied:

(1) $V$ is normal $\Q$-factorial and ${\rm Pic}\, (V) \otimes \Q = {\rm NS}\, (V) \otimes \Q$ under the cycle map.

(2) The nef cone $\overline{{\rm Amp}}\, (V)$ is generated by finitely many semi-ample divisor classes as a convex cone. 

(3) There are finitely many small $\Q$-factorial modifications $f_i : V \cdots\to V_i$ ($1 \le i \le n$, possibly $n=1$) such that $V_i$ satisfy (1) and (2) 
and
$${\rm Mov}\, (V) = \cup_{i=1}^{n} f_i^* \overline{{\rm Amp}}\, (V_i)\,\, .$$
\end{definition} 

The next theorem is the main result of this section:

\begin{theorem}\label{mds}
Let $S$ be a K3 surface such that $\vert {\rm Aut}\, (S) \vert < \infty$ and 
$\vert {\rm Aut}\, (S^{[n]}) \vert = \infty$. Then $S^{[n]}$ is not a MDS but $S^{(n)}$ is a MDS. In particular, the Hilbert-Chow morphism
$$\nu : S^{[n]} \rightarrow S^{(n)}$$
is a crepant projective resolution, which is also extremal in the sense of log minimal model program, but does not preserve MDS property. 
\end{theorem}

\begin{proof}
The fact that $\nu$ is a projective crepant extremal resolution is proved by Beauville (\cite[Prop. 5]{Be84}) based on a result of Forgaty. 

\begin{claim}\label{hilb} $S^{[n]}$ is not a MDS. 
More strongly, $\overline{{\rm Amp}}\, (S^{[n]})$ is not a finite rational polyhedral cone.
\end{claim}

\begin{proof}
We show the contraposition, i.e., assuming to the contrary that $\overline{{\rm Amp}}\, (S^{[n]})$ is a finite rational polyhedral cone, we prove $\vert {\rm Aut}\, (S^{[n]}) \vert < \infty$. 

By the assumption, there are only finitely many $1$-dimensional rays of the boundary of $\overline{{\rm Amp}}\, (S^{[n]})$. We denote all such rays by  $L_k$ ($1 \le k \le m$). Choose then the primitive integral generator $e_k$ of $L_k$. This is possible, as $L_k$ is defined over $\Q$, and $e_k$ is unique in each $L_k$. Consider the class
$$H := \sum_{k=1}^{m} e_k \in {\rm NS}\, (S^{[n]}) \simeq {\rm Pic}\, (S^{[n]}) \,\, .$$
As $\overline{{\rm Amp}}\, (S^{[n]})$ is generated by $e_k$ ($1 \le k \le m$) as a convex cone and each $e_k$ is nef, it follows from Kleiman's criterion that $H$ 
is an ample class. By the definition, the set $\{ e_k\}_{k=1}^{m}$ 
is stable under ${\rm Aut}\, (S^{[n]})$. Hence $H$ is invariant under ${\rm Aut}\, (S^{[n]})$ as well. Hence $\vert {\rm Aut}\, (S^{[n]}) \vert < \infty$ by 
\cite[Prop. 2.4]{Og12}. In order to apply \cite[Prop. 2.4]{Og12},  we also used the fact that $H^0(S^{[n]}, T_{S^{[n]}}) = 0$. 
\end{proof}
Let $S_k$ ($1 \le k \le n$) be projective K3 surfaces, possibly some of them are isomorphic, such that $\vert {\rm Aut}\,(S_k) \vert < \infty$ for each $k$. Set 
$$M := S_1 \times S_2 \times \cdots \times S_n\,\, .$$ 
\begin{proposition}\label{product}  
$M$ is a MDS. More strongly, $M$ satisfies that ${\rm Pic}\, (M) \simeq {\rm NS}\, (M)$, $\overline{{\rm Mov}}\, (M) = \overline{{\rm Amp}}\, (M)$, $\overline{{\rm Amp}}\, (M)$ is a finite rational polyhedral cone and any nef divisor of $M$ is semi-ample. 
\end{proposition}

\begin{proof} 
By Theorem \ref{elementary}, ${\rm Pic}\, (M) \simeq {\rm NS}\, (M)$, $\overline{{\rm Mov}}\, (M) = \overline{{\rm Amp}}\, (M)$. 

By $\vert {\rm Aut}\, (S_k) \vert < \infty$, the nef cone $\overline{{\rm Amp}}\, (S_k)$ is a finite rational polyhedral cone. This follows from the solution of the cone conjecture for projective K3 surfaces (\cite[Lem. 2.4]{St85}, \cite[Thm. 2.1]{Ka97}; see also \cite[Thm. 3.3, Cor. 5.1]{To10}, \cite[Thm. 2.7]{AHL10} for closely related results). Note that every nef divisor $D$ on a projective K3 surface $S$ is effective. Indeed, this follows from the Riemann-Roch formula. As $\omega_S \simeq \sO_S$, every nef divisor on $S$ is then semi-ample by the log abundance theorem in dimension $2$ (\cite{Fj84}). 

Hence, by Theorem \ref{elementary}, $\overline{{\rm Amp}}\, (M)$ is a finite rational polyhedral cone and every nef divisor on $M$ is semi-ample. 
\end{proof}

\begin{claim}\label{chow} $S^{(n)}$ is a MDS. 
\end{claim}

\begin{proof}
By Proposition \ref{product}, $S^{n}$ is a MDS. We have a surjective morphism 
$$S^{n} \rightarrow S^{(n)}\,\, ,\,\, (P_1, P_2, \cdots , P_n) \mapsto P_1+P_2 + \cdots + P_n\,\, .$$
Note that $S^{(n)} = S^n/\Sigma_n$ is a normal ${\mathbf Q}$-factorial projective variety. As $S^n$ is a MDS by Proposition \ref{product}, so is $S^{(n)}$ by a fundamental result of Okawa \cite[Thm. 1.1]{Ok11} (one may also apply \cite[Thm. 1.1]{Ba11} in our situation). 
\end{proof}

Claims \ref{hilb} and \ref{chow} complete the proof of Theorem \ref{mds}.
\end{proof}

Theorem \ref{main} (2) now follows from Theorem \ref{mds} and Theorem \ref{main} (1).


\begin{thebibliography}{BHPV}

\bibitem[AHL10]{AHL10}{Artebani, M.,  Hausen, J.,  Laface, A.,}  
{\em On Cox rings of K3 surfaces,} Compos. Math. {\bf 146}  (2010) 964--998. 

\bibitem[Ba11]{Ba11}{B\"aker, H.,}  {\em Good quotients of 
Mori dream spaces,} Proc. Amer. Math. Soc. {\bf 139}  (2011) 3135--3139. 

\bibitem[BHPV]{BHPV}{Barth, A., Hulek, K., Peters, C.,  Van de Ven, A., }  {\em Compact complex surfaces,} A Series of Modern Surveys in Mathematics, {\bf 4}, Springer-Verlag, Berlin, 2004.

\bibitem[Be83]{Be83}{Beauville, A.,}  {\em Some remarks on K\"ahler manifolds with $c_1 =0$,}  Classification of algebraic and analytic manifolds (Katata, 1982), 1--26, Progr. Math., {\bf 39} Birkh\"auser Boston, Boston, MA, 1983.

\bibitem[Be84]{Be84}{Beauville, A.,}  {\em Vari\'et\'es K\"ahleriennes dont la premi\`ere classe de Chern est nulle,}  J. Differential Geom. {\bf18}  (1984) 755--782. 

\bibitem[De84]{De84}{Debarre, O., }  {\em Un contre-exemple au th\'eor\`eme de Torelli pour les vari\'et\'es symplectiques irr\'eductibles,} C. R. Acad. Sci. Paris S\'er. I Math. {\bf 299}  (1984) 681--684. 

\bibitem[Do96]{Do96}{Dolgachev, I. V., }  {\em Mirror symmetry for lattice polarized K3 surfaces,} Algebraic geometry, 4. J. Math. Sci.  {\bf 81}  (1996) 2599--2630.

\bibitem[Fu83]{Fu83}{Fujiki, A.,}  {\em On primitively symplectic compact K\"ahler V -manifolds of dimension four.} Classification of algebraic and analytic manifolds (Katata, 1982),  71--250, Progr. Math., {\bf 39}, Birkh\"auser Boston, Boston, MA, 1983. 

\bibitem[Fj84]{Fj84}{Fujita, T.,}  {\em Fractionally logarithmic canonical rings of algebraic surfaces,}  J. Fac. Sci. Univ. Tokyo Sect. IA Math.  
{\bf 30}  (1984) 685--696. 

\bibitem[Ka84]{Ka84}{Kawamata, Y.,} {\em The cone of curves of algebraic varieties,} Ann. of Math.  {\bf 119}  (1984) 603--633.

\bibitem[Ka88]{Ka88}{Kawamata, Y.,} {\em Crepant blowing-up of $3$-dimensional canonical singularities and its application to degenerations of surfaces,} Ann. of Math.   {\bf 127}  (1988) 93--163.

\bibitem[Ka97]{Ka97}
{Kawamata, Y.,} {\em On the cone of divisors of Calabi-Yau fiber 
spaces,} Internat. J. Math.  {\bf 8}  (1997) 665--687.

\bibitem[Ka08]{Ka08}
{Kawamata, Y.,} {\em Flops connect minimal models,} 
Publ. Res. Inst. Math. Sci.  {\bf 44}  (2008) 419--423. 

\bibitem[HK00]{HK00} {Hu, Y., Keel, S.,} {\em Mori dream spaces and GIT,} 
Dedicated to William Fulton on the occasion of his 60th birthday, 
Michigan Math. J. {\bf 48} (2000) 331--348. 

\bibitem[MN02]{MN02} {Madonna, C., Nikulin, V.V.,} {\em Classical correspondence between K3 surfaces}, arXiv:0206158.
 
\bibitem[MY14]{MY14} {Markman, E., Yoshioka, K.,} {\em A proof of the Kawamata-Morrison Cone Conjecture for holomorphic symplectic varieties of $K3^{[n]}$ or generalized Kummer deformation type,} arXiv:1402.2049.

\bibitem[Mo84]{Mo84} {Morrison, D. R.,}  {\em On K3  surfaces with large Picard number,} Invent. Math.  {\bf 75}  (1984) 105--121. 

\bibitem[Og14]{Og12}{Oguiso, K.,} {\em Automorphism groups of Calabi-Yau manifolds of Picard number two,} J. Algebraic Geom.  {\bf 23}  (2014) 775--795. 

\bibitem[OGr05]{OGr05}{O'Grady, K. G.,} {\em Involutions and linear systems on holomorphic symplectic manifolds,} Geom. Funct. Anal.  {\b 15}  (2005),  no. 6, 1223--1274. 

\bibitem[Ok11]{Ok11}{Okawa, S.,} {\em On images of Mori dream spaces,} 
arXiv:1104.1326.

\bibitem[SD74]{SD74} {Saint-Donat, B.,} {\em Projective models of K?3  surfaces,} Amer. J. Math. {\bf 96}  (1974), 602--639. 

\bibitem[St85]{St85} {Sterk, H.,} {\em  Finiteness results for algebraic K3 surfaces,} Math. Z.  {\bf 189}  (1985),  no. 4, 507--513.

\bibitem[To10]{To10} {Totaro, B.,} {\em The cone conjecture for Calabi-Yau pairs in dimension 2,} Duke Math. J.  {\bf 154}  (2010),  no. 2, 241--263.
\end{thebibliography}
\end{document}